\documentclass[12pt]{amsart}
\usepackage{pgf,tikz,pgfplots} 
\pgfplotsset{compat=1.14}
\usepackage{mathrsfs}
\usetikzlibrary{arrows}
\usetikzlibrary{patterns}
\usepackage{pgfplots}
\usetikzlibrary{intersections, pgfplots.fillbetween}
\usepackage[colorlinks=true,urlcolor=blue, citecolor=red,linkcolor=blue,linktocpage,pdfpagelabels, bookmarksnumbered,bookmarksopen]{hyperref}
\usepackage[hyperpageref]{backref}
\usepackage{cleveref}
\usepackage{xcolor}
\usepackage{amsthm} 
\usepackage{latexsym,amsmath,amssymb}
\usepackage{accents}
\usepackage[colorinlistoftodos,prependcaption,textsize=tiny]{todonotes}
\usepackage{a4wide}
\usepackage{soul}
\usepackage{mathtools} 
\usepackage{xparse} 

\pgfdeclarelayer{ft}
\pgfdeclarelayer{bg}
\pgfsetlayers{bg,main,ft}

\title[Failure of Calder\'on-Zygmund estimates for the {\MakeLowercase{p}}-Laplace equation]{Failure of $L^r$-Calder\'on-Zygmund estimates for the {\MakeLowercase{p}}-Laplace equation for small $r$}

\author{Armin Schikorra}
\address[Armin Schikorra]{Department of Mathematics,
University of Pittsburgh,
301 Thackeray Hall,
Pittsburgh, PA 15260, USA}
\email{armin@pitt.edu}
\definecolor{indigo}{rgb}{0.29, 0.0, 0.51}
\definecolor{p1}{gray}{0.4}
\definecolor{p2}{gray}{0.6}
\definecolor{p3}{gray}{0.98}
\definecolor{p4}{gray}{0.8}
\definecolor{p5}{gray}{0.9}

plus 6pt minus 12pt
\def\eps{\varepsilon}

\def\B{\mathbb{B}}

\def\N{{\mathbb N}}

\renewcommand{\div}{{\rm div}}

\newtheorem{theorem}{Theorem}
\newtheorem{lemma}[theorem]{Lemma}

\newtheorem{proposition}[theorem]{Proposition}

\newtheorem{remark}[theorem]{Remark}
\newtheorem{definition}[theorem]{Definition}

\newtheorem{conjecture}[theorem]{Conjecture}

\def\dist{{\rm dist\,}}

\def\lip{{\rm Lip\,}}
\def\rank{{\rm rank\,}}
\def\supp{{\rm supp\,}}

\newcommand{\R}{\mathbb{R}}

\newcommand{\brac}[1]{\left (#1 \right )}
\newcommand{\abs}[1]{\left |#1 \right |}

\newcommand{\barint}{
\rule[.036in]{.12in}{.009in}\kern-.16in \displaystyle\int }

\newcommand{\barcal}{\mbox{$ \rule[.036in]{.11in}{.007in}\kern-.128in\int $}}

\def\mvint_#1{\mathchoice
          {\mathop{\vrule width 6pt height 3 pt depth -2.5pt
                  \kern -8pt \intop}\nolimits_{\kern -3pt #1}}%
          {\mathop{\vrule width 5pt height 3 pt depth -2.6pt
                  \kern -6pt \intop}\nolimits_{#1}}%
          {\mathop{\vrule width 5pt height 3 pt depth -2.6pt
                  \kern -6pt \intop}\nolimits_{#1}}%
          {\mathop{\vrule width 5pt height 3 pt depth -2.6pt
                  \kern -6pt \intop}\nolimits_{#1}}}

\numberwithin{theorem}{section} \numberwithin{equation}{section}

\newcommand{\lap}{\Delta }
\newcommand{\aleq}{\precsim}
\newcommand{\ageq}{\succsim}
\newcommand{\aeq}{\asymp}

\let\latexchi\chi
\makeatletter
\renewcommand\chi{\@ifnextchar_\sub@chi\latexchi}
\newcommand{\sub@chi}[2]{
  \@ifnextchar^{\subsup@chi{#2}}{\latexchi_{#2}}%
}
\newcommand{\subsup@chi}[3]{
  \latexchi_{#1}^{#3}%
}
\makeatother

\makeatletter
\def\tikz@arc@opt[#1]{
  {%
    \tikzset{every arc/.try,#1}%
    \pgfkeysgetvalue{/tikz/start angle}\tikz@s
    \pgfkeysgetvalue{/tikz/end angle}\tikz@e
    \pgfkeysgetvalue{/tikz/delta angle}\tikz@d
    \ifx\tikz@s\pgfutil@empty%
      \pgfmathsetmacro\tikz@s{\tikz@e-\tikz@d}
    \else
      \ifx\tikz@e\pgfutil@empty%
        \pgfmathsetmacro\tikz@e{\tikz@s+\tikz@d}
      \fi%
    \fi
    \tikz@arc@moveto
    \xdef\pgf@marshal{\noexpand%
    \tikz@do@arc{\tikz@s}{\tikz@e}
      {\pgfkeysvalueof{/tikz/x radius}}
      {\pgfkeysvalueof{/tikz/y radius}}}%
  }%
  \pgf@marshal%
  \tikz@arcfinal%
}
\let\tikz@arc@moveto\relax
\def\tikz@arc@movetolineto#1{%
  \def\tikz@arc@moveto{\tikz@@@parse@polar{\tikz@arc@@movetolineto#1}(\tikz@s:\pgfkeysvalueof{/tikz/x radius} and \pgfkeysvalueof{/tikz/y radius})}}
\def\tikz@arc@@movetolineto#1#2{#1{\pgfpointadd{#2}{\tikz@last@position@saved}}}
\tikzset{%
  move to start/.code=\tikz@arc@movetolineto\pgfpathmoveto,%
  line to start/.code=\tikz@arc@movetolineto\pgfpathlineto}
\makeatother

\newcommand{\tnn}{\aleph} 

\newcommand{\tnell}{\mathfrak{l}}
\newcommand{\tni}{\mathfrak{i}}
\newcommand{\tnj}{\mathfrak{j}}
\newcommand{\tnk}{\mathfrak{k}}

\begin{document}
\begin{abstract}
Let $p \neq 2$. For any small enough $r> \max \{p-1,1\}$ and for any $\Lambda > 1$ there exists a Lipschitz function $u$ and a bounded vectorfield $f$ such that
\[
\begin{cases}
 \div(|\nabla u|^{p-2} \nabla u) = \div (f) \quad& \text{in $\mathbb{B}^2$}\\
u=0 &\text{on $\partial \mathbb{B}^2$}
\end{cases}
 \]
but
\[
 \int_{\mathbb{B}^2} |\nabla u|^r \not \leq  \Lambda \int_{\mathbb{B}^2} |f|^{\frac{r}{p-1}}.
\]
This disproves a conjecture by Iwaniec from 1983. The proof adapts recent convex-integration ideas by Colombo--Tione.
\end{abstract}
\maketitle
%

\section{Introduction}
Classical Calder\'on-Zygmund theory states that if we have a variational solution $u \in W^{1,2}(\Omega)$ which solves in distributional sense
\[
 \begin{cases} \lap u = \div (f) \quad &\text{in a smooth domain $\Omega \subset \R^n$}\\
  u = 0 \quad &\text{on $\partial \Omega$}
 \end{cases}
\]
then for any $r \in (1,\infty)$
\[
 \|\nabla u\|_{L^r(\Omega)} \aleq \|f\|_{L^r(\Omega)}
\]
For small $r$, the analogue estimate where the linear elliptic equation $\lap u = \div (f)$ is replaced with the $p$-Laplace equation $\div(|\nabla u|^{p-2} \nabla u)= \div(f)$ is not known. Indeed, in 1983 Iwaniec \cite[pages 294-295]{I83} posed the following conjecture regarding, see also \cite[page 36]{DOI05}.

\begin{conjecture}[Iwaniec]\label{conjv1}
For any $r > \max\{p-1,1\}$, $p > 1$, and $n \geq 2$ there exists a constant $C_{r,p,n}$ such that whenever $u \in W^{1,r}_0(\B^n)\cap W^{1,p}_0(\B^n)$ and $f \in L^{\frac{r}{p-1}}\cap L^{p'}(\B^n,\R^n)$ (where $\B^n$ is the $n$-unit ball) solves
\begin{equation}\label{eq:pdeconj}
 \div(|\nabla u|^{p-2} \nabla u) = \div(f) \quad \text{in $\B^n$}
\end{equation}
then
\begin{equation}\label{eq:ineq}
 \int_{\B^n} |\nabla u|^r \leq C_{r,p,n} \int_{\B^n} |f|^{\frac{r}{p-1}}.
\end{equation}
\end{conjecture}
As for positive results, it is known that there exists some $r_0 < p$ such that \eqref{eq:ineq} is true for $r \in (r_0,\infty)$ and for $BMO$, see in particular \cite{I83,IM89,DiBM93,DKS12}. For a recent overview over known results of Calderon-Zygmund type estimates we refer to the exposition in \cite{BDW20}. Let us also remark, that if $\div f$ on the right-hand side of \eqref{eq:pdeconj} is replaced with a measure $\mu$ then Calder\`on-Zygmund theory \cite{M07} and pointwise potential estimates hold for $|\nabla u|^{p-1}$ -- notably even in the vectorial case, see \cite{KM18} and references within, which might lead one to the (false) hope that \eqref{conjv1} could indeed be true also for small $r>p-1$.

Related to \Cref{conjv1}, in 1994 Iwaniec and Sbordone \cite{IS94} posed the following problem regarding the \emph{regularity theory for very weak solutions}.
\begin{conjecture}[Iwaniec-Sbordone]\label{conjIS}
For any $r > p-1$ assume $u \in W^{1,r}(B)$ solves in distributional sense
\[
 \div(|\nabla u|^{p-2} \nabla u) = 0 \quad \text{in $B$}
\]
then $u \in W^{1,p}_{loc}$.
\end{conjecture}
\Cref{conjv1} and \Cref{conjIS} are clearly related, but neither implies the other in an obvious way. Observe that in \Cref{conjv1} we crucially assume that $u$ already belongs to the energy space $W^{1,p}$, i.e. we are only interested in the (unique) variational solution.

It turns out both conjectures are wrong. Indeed, quite recently, \Cref{conjIS} was disproved by Colombo and Tione \cite{CT22} with methods from convex integration. Here we disprove \Cref{conjv1} and to the contrary establish the following for the unit ball $\B^2$.
\begin{theorem}\label{th:main}
Let $p\neq 2$. There exists $\bar{q} > \max\{p-1,1\}$ such that the following holds for any $\Lambda > 0$ and $r \in (\max\{p-1,1\},\bar{q})$.

There exists $u,v \in W^{1,p}_0\cap \lip(\B^2)$ (even piecewise affine) such that
\begin{equation}\label{eq:goal}
 \int_{\B^2} |\nabla u|^r > \Lambda \int_{\B^2} \abs{|\nabla u|^{p-2} \nabla u - \nabla^\perp v}^{\frac{r}{p-1}}.
\end{equation}
\end{theorem}
Here $\nabla^\perp = (-\partial_2, \partial_1)^T$ denotes the rotated gradient. Observe that
\[
 \div(|\nabla u|^{p-2} \nabla u) = \div(|\nabla u|^{p-2} \nabla u -\nabla^\perp v)
\]
so if we set $f:=|\nabla u|^{p-2} \nabla u -\nabla^\perp v$ for $u$ and $v$ as in \Cref{th:main} we have found a counterexample to \Cref{conjv1}.

Let us remark that our threshold value for $\bar{q}$ in \Cref{th:main} is explicit, and it is the same one as for the counterexample to \Cref{conjIS}  obtained in \cite{CT22}. Precisely, \Cref{th:main} holds for any
\[
 \max\{p-1,1\} < \bar{q} < \sup_{b \in (0,\infty)} \frac{p-1}{b^{p-1}+1} + \frac{b}{b+1}.
\]
We observe that for $p=2$ there is no such $\bar{q}$. But for $p \neq 2$ we can find $\bar{q}$ as above. Of course it would be very interesting to see if this is sharp, i.e. if \Cref{conjv1} holds for any $r > \sup_{b \in (0,\infty)} \frac{p-1}{b^{p-1}+1} + \frac{b}{b+1}$.

\subsection*{Outline} We prove \Cref{th:main} adapting the crucial insights of \cite{CT22} to our situation. The idea is to utilize a trade-off: In \cite{CT22} they find a solution to $\div(|\nabla u|^{p-2}\nabla u) = 0$, but it should not belong to $W^{1,r}$. Here we want to find a solution which is Lipschitz, but we are happy to have a solution to the inhomogeneous equation $\div(|\nabla u|^{p-2}\nabla u) =\div(f)$ -- where we have a choice of $f$. For this, in \Cref{s:laminates} we recall the notion of laminates of finite orders and the convex integration result of constructing solutions adapted to the support of these laminates. In \Cref{s:constrlaminate} we construct an explicit laminate to utilize. In \Cref{s:proofofthm} we then use the convex integration solution subordinated to this laminate to prove \Cref{th:main}.

\subsection*{Acknowledgment} The author is very grateful to Phuc Cong Nguyen for discussions regarding the Iwaniec conjecture \cite{I83}, and Juan Manfredi for helpful comments. Part of this work was carried out while the author was visiting Chulalongkorn University. A.S. is an Alexander-von-Humboldt Fellow. A.S. is funded by NSF Career DMS-2044898.

\section{Laminates of finite order}\label{s:laminates}
We recall the definition of laminates of finite order, see e.g. \cite[Definition 2.2.]{AFS08}.
\begin{definition}[Laminate of finite order and elementary splittings]
The family of laminates of finite order $\mathcal{L}(\R^{2 \times 2})$ is the smallest family of probability measures in $\mathcal{M}(\R^{2 \times 2})$ with the properties
\begin{enumerate}
 \item Any dirac measure $\delta_{A} \in \mathcal{L}(\R^{2 \times 2})$ for any $A \in \R^{2 \times 2}$.
 \item If $\sum_{\tni=1}^{\tnn} \lambda_\tni \delta_{A_\tni} \in \mathcal{L}(\R^{2 \times 2})$ and $A_1 = \lambda B + (1-\lambda) C$ for some $\lambda \in [0,1]$ and $\rank (B-C) =1$. Then the probability measure
 \[
  \sum_{\tni=2}^{\tnn} \lambda_\tni \delta_{A_\tni}  + \lambda_1 (\lambda \delta_{B} + (1-\lambda) \delta_C)
 \]
is also contained in $\mathcal{L}(\R^{2 \times 2})$. The operation above is called elementary splitting, and if we can obtain a laminate of finite order $\tilde{\mu}$ out of a measure $\mu$ via finitely many elementary splittings we write
\[
 \mu \xrightarrow{\text{elementary splitting}} \tilde{\mu}.
\]

\end{enumerate}
\end{definition}

The main property of laminates of finite order $\sum_{\tni=1}^{\tnn} \lambda_\tni \delta_{A_\tni}$ that we are going to use is: we can find functions $f$ whose gradient $\nabla f$ is close to $A_{\tni}$ exactly with probability $\lambda_{\tni}$. This was proven in \cite[Lemma 3.2]{MS03}, see also \cite[Proposition 2.3.]{AFS08}. We formulate their result with the properties of relevance to us, for related results and a literature overview see also \cite[Section 2.3]{KMSX24}.
\begin{proposition}\label{pr:wigglelam}
Let $\mu = \sum_{\tni=1}^{\tnn} \lambda_\tni \delta_{A_\tni}$ be a laminate of finite order in $\R^{2 \times 2}$ with baricenter
\[
 A:= \overline{\mu} \equiv \sum_{\tni=1}^{\tnn} \lambda_{\tni} A_{\tni} \in \R^{2 \times 2}
\]
Then for any $\delta < \min_{\tni \neq \tnj} \frac{|A_\tni-A_{\tnj}|}{2}$ and any smooth, open, and bounded set $\Omega \subset \R^2$ there exists a piecewise affine Lipschitz map $w: \Omega \to \R^2$ such that
\begin{itemize}
 \item $w(x) = Ax$ on $\partial \Omega$
 \item $|\{x \in \Omega: \dist(\nabla w, A_{\tni}) < \delta\}| = \lambda_{\tni} |\Omega|$ for each $\tni=1,\ldots,\tnn$
 \item $\dist(\nabla w,\supp \mu) < \delta$ a.e.
\end{itemize}
\end{proposition}

\section{Construction of a laminate}\label{s:constrlaminate}
We will construct a laminate of finite order explicitly. It is strongly inspired by the argument used in \cite{CT22}, see also the presentation in \cite{KMSX24}.

First we define our threshold exponent $\bar{q}_1$.
\begin{equation}\label{q1}
 \bar{q}_1 := \sup_{b \in (0,\infty)} \frac{p-1}{b^{p-1}+1} + \frac{b}{b+1}
\end{equation}

The crucial observation is that $\bar{q}_1 > \max \{p-1,1\}$ if $p \neq 2$. Indeed, a simple computation yields the following,
\begin{lemma}\label{la:defq1}
Take $\bar{q}_1$ from \eqref{q1}.
Then
\begin{itemize}
 \item If $p=2$ then $\bar{q}_1 = 1$.
 \item If $p > 2$ then $\bar{q}_1 > p-1$
 \item If $1<p < 2$ then $\bar{q}_1 > 1$
\end{itemize}
\end{lemma}
For the convenience of the reader we discuss the
\begin{proof}
The case $p=2$ is obvious.

Assume now \underline{$p>2$}. We show that for small $b > 0$ we have
\[
  \frac{p-1}{b^{p-1}+1} + \frac{b}{b+1} > p-1.
\]
Observe
\[
\begin{split}
&  \frac{p-1}{b^{p-1}+1} + \frac{b}{b+1} > p-1\\
\Leftrightarrow&  \frac{b}{b+1} + (p-1)\brac{\frac{1}{b^{p-1}+1} -1} >0\\
\Leftrightarrow& \frac{b}{b+1} - (p-1)\frac{b^{p-1}}{b^{p-1}+1} >0\\
\Leftrightarrow& 1 > (p-1)\frac{b^{p-1} +b^{p-2}}{b^{p-1}+1} \\
\end{split}
 \]
Since $p>2$ we see that for small enough $b>0$ the last inequality is true, and we can conclude.

If \underline{$1<p < 2$}
\[
\begin{split}
&  \frac{p-1}{b^{p-1}+1} + \frac{b}{b+1} > 1\\
\Leftrightarrow& \frac{p-1}{b^{p-1}+1} > \frac{1}{b+1} \\
\Leftrightarrow& (p-1) > \frac{b^{p-1}+1}{b+1} \\
\end{split}
 \]
Since $p-1 \in (0,1)$ we see that $\lim_{b \to \infty} \frac{b^{p-1}+1}{b+1} =0$ and thus the last inequality is satisfied for all large $b \gg 1$. We can conclude.
\end{proof}

Now we are going to define our laminate of finite order. For $\tni =0,1,2,\ldots$, $p \in (1,\infty)$, and $b \in (0,\infty)$, $b \neq 1$, we consider sequences of $\R^{2 \times 2}$-matrices
\[
 A_\tni = \left ( \begin{array}{cc}
                b \tni & 0\\
                0 & \tni^{p-1}
               \end{array}
\right ), \quad \tni = 0,1,2,\ldots
\]
\[
 B_{\tni} = \left ( \begin{array}{cc}
                b(\tni-1) & 0\\
                0 & {-}b^{p-1} (\tni-1)^{p-1}
               \end{array}
\right ), \quad \tni = 1,2,\ldots
\]
and
\[
 C_{\tni} = \left ( \begin{array}{cc}
                {-} \tni& 0\\
                0 & \tni^{p-1}
               \end{array}
\right ), \quad \tni = 1,2,\ldots
\]
Clearly $\pm C_{\tni}, \pm B_{\tnj}, \pm A_{\tnell}$ are all mutually distinct since $b \neq 1$, and indeed we have a uniform minimal distance between them,
\begin{equation}\label{eq:mindist}
 \inf_{\tnn \geq 2} \inf_{A \neq \tilde{A} \in \{\pm C_{\tni}, \pm B_{\tnj}:\ \tni \in \{2,\ldots,\tnn\}\} \cup \{\pm A_{\tnn}\} }|\tilde{A}-A| \ageq_{b} 1
\end{equation}

The role of $B_{\tni}, C_{\tni}$ is as follows: if for any $\tni \geq 1$
\begin{equation}\label{eq:nablauinKp}
 \left (\begin{array}{cc}
  \partial_1 u & \partial_2 u\\
  \partial_1 v & \partial_2 v
 \end{array} \right) = \pm C_{\tni} \text{ or } \pm B_{\tni} \text{ then } |\nabla u|^{p-2} \nabla u - \nabla^\perp v = 0.
\end{equation}
The point of the latter is that this happens even though $|\nabla u| \neq 0$, i.e. \eqref{eq:goal} fails when the integrals are restricted to sets where \eqref{eq:nablauinKp} holds. The idea of our laminate is thus to ensure this situation \eqref{eq:nablauinKp} happens on a large part of the domain, and estimate the remainder to arrive at \eqref{eq:goal}.

We collect the relevant properties our laminate of finite order in the following.

\begin{proposition}\label{pr:thelaminate}
Let $p \in (1,\infty)$, and take $\bar{q}_1$ from \eqref{q1}. For any $\bar{q} \in (0,\bar{q}_1)$ there exists $b \in (0,\infty)$, $b \neq 1$, such that for any $\tnn \geq 2$ the following holds:

There exists sequences $(\bar{\alpha}_{\tni})_{\tni=2}^{\tnn}$, $(\bar{\beta}_{\tni})_{\tni=2}^{\tnn}$, $\bar{\Gamma}_\tnn$ all in $(0,1)$ such that
\begin{itemize} \item The following is a laminate of finite order with baricenter $0 \in \R^{2 \times 2}$
\[
 \sum_{\tni=2}^{\tnn} \brac{\bar{\alpha}_\tni \delta_{B_\tni} +\bar{\alpha}_\tni \delta_{-B_\tni} + \bar{\beta}_\tni \delta_{C_\tni}+\bar{\beta}_\tni \delta_{-C_\tni}} + \bar{\Gamma}_{\tnn} \delta_{A_\tnn}+\bar{\Gamma}_{\tnn} \delta_{-A_\tnn}
\]
\item We have
\[
 \bar{\Gamma}_{\tnn} \aleq_{p,b} \tnn^{-\bar{q}} \quad \forall \tnn \geq 1.
\]
\end{itemize}
\end{proposition}
\Cref{pr:thelaminate} follows from \Cref{la:elemsplit}, \Cref{la:laminatestep0}, and \Cref{la:gammaest} below. More precisely, the first part of \Cref{pr:thelaminate} follows from an induction argument via the following construction, combined with the initial step in \Cref{la:laminatestep0}.
\begin{lemma}\label{la:elemsplit}
For $\tni \geq 1$ We have via two elementary splittings
\[
 \delta_{\pm A_\tni} \xrightarrow{\text{elementary splitting}}\alpha_{\tni+1} \delta_{\pm B_{\tni+1}} + \beta_{\tni+1} \delta_{\pm C_{\tni+1}} + \gamma_{\tni+1} \delta_{\pm A_{\tni+1}},
\]
with values
\[
 \alpha_{\tni+1} = \frac{(\tni+1)^{p-1}-\tni^{p-1}}{(\tni+1)^{p-1}+b^{p-1} \tni^{p-1}}
 = \frac{(1+1/\tni)^{p-1}-1}{(1+1/\tni)^{p-1}+b^{p-1}}=1-\frac{1+b^{p-1}}{(1+1/\tni)^{p-1}+b^{p-1}}
 \]
\[
  \beta_{\tni+1} = (1-\alpha_{\tni+1})\frac{b(\tni+1)-b\tni}{b(\tni+1)+(\tni+1)} = \frac{1+b^{p-1}}{(1+1/\tni)^{p-1}+b^{p-1}}\frac{b}{(b+1)} \frac{1}{\tni+1}
\]
and
\[
 \gamma_{\tni+1} = 1-\alpha_{\tni+1}-\beta_{\tni+1}.
 \]
\end{lemma}
\begin{proof}
Let $\tni \geq 1$, we obtain the first elementary splitting via the rank-one convex splitting
\[
\begin{split}
A_{\tni} =  \left ( \begin{array}{cc}
                b\tni & 0\\
                0 & \tni^{p-1}
               \end{array}
\right ) =& \alpha \left ( \begin{array}{cc}
                b\tni & 0\\
                0 & {-}b^{p-1} \tni^{p-1}
               \end{array}
\right ) + (1-\alpha) \left ( \begin{array}{cc}
                b\tni & 0\\
                0 & (\tni+1)^{p-1}
               \end{array}
\right )\\
=&\alpha B_{\tni+1} + (1-\alpha) \left ( \begin{array}{cc}
                b\tni & 0\\
                0 & (\tni+1)^{p-1}
               \end{array}
\right ).
\end{split}
\]
Here we have
\[
 \alpha = \frac{(\tni+1)^{p-1}-\tni^{p-1}}{(\tni+1)^{p-1}+b^{p-1} \tni^{p-1}} \in (0,1).
\]
The second elementary splitting is from the rank-one convex splitting
\[
\begin{split}
\left ( \begin{array}{cc}
                b\tni & 0\\
                0 & (\tni+1)^{p-1}
               \end{array}
\right )
=& \beta \left ( \begin{array}{cc}
                {-} (\tni+1)& 0\\
                0 & (\tni+1)^{p-1}
               \end{array}
\right ) + (1-\beta) \left ( \begin{array}{cc}
                b(\tni+1) & 0\\
                0 & (\tni+1)^{p-1}
               \end{array}
\right )\\
=&\beta C_{\tni+1} + (1-\beta) A_{\tni+1}
\end{split}
\]
and we have
\[
 \beta = \frac{b(\tni+1)-b\tni}{b(\tni+1)+(\tni+1)}
\]
\end{proof}

For $A_0 = 0 \in \R^{2 \times 2}$ we have to adapt the splitting
\begin{lemma}\label{la:laminatestep0}
We have
\[
 \delta_{A_0} \xrightarrow{\text{elementary splitting}} \frac{1}{4} \delta_{B_2} + \frac{1}{4} \delta_{-B_2} + \frac{1}{4} \delta_{A_1} + \frac{1}{4} \delta_{-A_1}
\]
\end{lemma}
\begin{proof}
We use two elementary splittings
\[
\begin{split}
A_{0} =  \left ( \begin{array}{cc}
                0 & 0\\
                0 & 0
               \end{array}
\right )
=& \frac{1}{2} \left ( \begin{array}{cc}
                b & 0\\
                0 & 0
               \end{array}
\right ) + \frac{1}{2}\left ( \begin{array}{cc}
                -b & 0\\
                0 & 0
               \end{array}
\right )
\end{split},
\]
and
\[
\begin{split}
 \pm\left ( \begin{array}{cc}
                b & 0\\
                0 & 0
               \end{array}
\right ) =&  \pm \frac{1}{2}\left ( \begin{array}{cc}
                b & 0\\
                0 & -b^{p-1}
               \end{array}
\right )
 \pm \frac{1}{2}\left ( \begin{array}{cc}
                b & 0\\
                0 & b^{p-1}
               \end{array}
\right )\\
=&\frac{1}{2} \brac{\pm B_2} + \frac{1}{2} (\pm A_1)
\end{split}
\]
\end{proof}

\begin{lemma}\label{la:gammaest}
For any $p \in (1,\infty)$ and any $\bar{q} \in (0,\bar{q}_1)$ where $\bar{q}_1$ is from \eqref{q1} there exists $b \in (0,\infty)$, $b \neq 1$, with the following properties:

Let $(\gamma_{\tni})_{\tni=1}^{\tnn}$ from \Cref{la:elemsplit}.
Then we have
\begin{equation}\label{eq:betan}
 \Gamma_{\tnn} := \Pi_{\tnj=1}^\tnn \gamma_\tnj \aleq_{p,b,\bar{q}} \tnn^{-\bar{q}} \quad \text{for all } \tnn \geq 1.
\end{equation}
\end{lemma}
\begin{proof}
From \Cref{la:elemsplit} we have
\[
 \gamma_{\tnk+1} = \frac{1+b^{p-1}}{(1+1/\tnk)^{p-1}+b^{p-1}}(1-\frac{b}{b+1} \frac{1}{\tnk+1})
 \]
and thus
\[
 -\log \gamma_{\tnk+1} = -\log \frac{1+b^{p-1}}{(1+1/\tnk)^{p-1}+b^{p-1}} - \log (1-\frac{b}{b+1} \frac{1}{\tnk+1}).
\]
By a Taylor expansion we see
\[
 -\log \frac{1+b^{p-1}}{(1+1/\tnk)^{p-1}+b^{p-1}} = \frac{1}{\tnk} \frac{p-1}{b^{p-1}+1} + O_b(\frac{1}{\tnk^2}) \quad \text{as $\tnk \to \infty$}.
 \]
Similarly,
\[
 - \log (1-\frac{b}{(b+1)} \frac{1}{\tnk+1}) = \frac{1}{\tnk } \frac{b}{b+1} + O_b(\frac{1}{\tnk^2}) \quad \text{as $\tnk \to \infty$}
\]
and thus if we set
\[
 \tilde{q}_b := \frac{p-1}{b^{p-1}+1} + \frac{b}{b+1}
\]
then we have
\[
 -\log \gamma_{\tnk+1} = \frac{1}{\tnk} \tilde{q}_b + O_b(\frac{1}{\tnk^2})
\]
In particular, by the definition of  $\bar{q}_1$ in \eqref{q1}, if $\bar{q} < \bar{q}_1$ we find some $b >0$ such that $\bar{q} < \tilde{q}_b$, and there exists some $\tnk_0 \geq 1$ such that
\[
 \gamma_{\tnk+1} \leq e^{-\frac{1}{\tnk} \bar{q}} \quad \forall \tnk \geq \tnk_0.
\]
Thus
\[
 \Pi_{\tni=1}^{\tnn} \gamma_{\tni} \leq \Pi_{\tni=\tnk_0}^{\tnn} \gamma_{\tni} \leq e^{-\bar{q}\sum_{\tnk = \tnk_0}^{\tnn} \frac{1}{\tnk}}\leq e^{-\bar{q} \log (\tnn - \tnk_0) } = (\tnn-\tnk_0)^{-\bar{q}}
\]
where we have used
\[
 \sum_{n=1}^\tni \frac{1}{n} \geq \log \tni.
\]
We thus find
\[
 \Pi_{\tni=1}^{\tnn} \gamma_{\tni} \leq \Pi_{\tni=\tnk_0}^{\tnn} \gamma_{\tni} \leq e^{-\bar{q}\sum_{\tnk = \tnk_0}^{\tnn} \frac{1}{\tnk}}\leq e^{-\bar{q} \log (\tnn - \tnk_0) } \aleq_{q} \tnn^{-\bar{q}} \quad \forall \tnn \geq 2 \tnk_0
\]
Since the statement
\[
 \Pi_{\tni=1}^{\tnn} \gamma_{\tni}  \aleq \tnn^{-\bar{q}} \quad \tnn = 1,\ldots,2\tnk_0
\]
is trivial we can conclude.
\end{proof}

\section{Proof of Theorem~\ref{th:main}}\label{s:proofofthm}
Fix $\Lambda > 0$, take $\bar{q} \in (\max\{p-1,1\},\bar{q}_1)$ where $\bar{q}_1$ is from \eqref{q1}. By \Cref{la:defq1} such a $\bar{q}$ exists.

We will assume for now $p>2$. We indicate the minor changes of the argument for $p \in (1,2)$ at the end of the proof.

For some $\tnn \gg 1$ yet to be chosen in dependence of $\Lambda$, take the laminate of finite order from \Cref{pr:thelaminate}, together with $b \in (0,\infty)$, $b \neq 1$.
For a $\delta > 0$ also yet to be chosen in dependence of $\Lambda$, we apply \Cref{pr:wigglelam} and obtain a piecewise Lipschitz map $w =: \left ( \begin{array}{c} u\\v\end{array}\right): \B^2 \to \R^2$.

We claim, that $\tnn$ suitably large and $\delta$ suitably small (both in dependence of $\Lambda$) we have
\begin{equation}\label{eq:claim}
 \int_{\B^2} |\nabla u(x)|^r > \Lambda \int_{\B^2} \abs{|\nabla u(x)|^{p-2} \nabla u(x) -\nabla^\perp v(x)}^{\frac{r}{p-1}}.
\end{equation}
which would prove \Cref{th:main}.

Assume, by contradiction, that \eqref{eq:claim} is false, i.e. assume
\begin{equation}\label{eq:claimbycontradiction}
 \int_{\B^2} |\nabla u(x)|^r \leq \Lambda \int_{\B^2} \abs{|\nabla u(x)|^{p-2} \nabla u(x) -\nabla^\perp v(x)}^{\frac{r}{p-1}}.
\end{equation}

We introduce the notation for $A \in \R^{2 \times 2}$
\[
 \Omega_{A} := \left \{x \in \B^2: \dist\brac{\nabla \brac{\begin{array}{c} u\\ v\end{array} }(x), A} < \delta \right \},
\]
which is well-defined up to a negligible zero-set.
From the properties of \Cref{pr:wigglelam} we have (assuming $\delta$ is suitably small, where in view of \eqref{eq:mindist} this smallness is independent of $\tnn$) for a zeroset $N$ we have the disjoint decomposition
\begin{equation}
 \B^2 = N \dot{\cup} \Omega_{A_{\tnn}} \dot{\cup} \Omega_{-A_{\tnn}} \dot{\cup} \dot{\bigcup}_{\tni=2}^{\tnn} \brac{\Omega_{B_{\tni}} \dot{\cup} \Omega_{-B_{\tni}} \dot{\cup} \Omega_{C_{\tni}} \dot{\cup} \Omega_{-C_{\tni}}}.
\end{equation}
Moreover we have precise size estimates,
\begin{equation}\label{eq:Annguy}
 |\Omega_{\pm A_{\tnn}}| = \bar{\Gamma}_{\tnn}|\B^2|,
\end{equation}
\[
 |\Omega_{\pm B_{\tni}}| = \bar{\alpha}_{\tni} |\B^2|,
\]
\[
 |\Omega_{\pm C_{\tni}}| = \bar{\beta}_{\tni} |\B^2|.
\]

The idea is simple: Estimate $\abs{|\nabla u|^{p-2} \nabla u - \nabla^\perp v}^{\frac{r}{p-1}}$ and $|\nabla u|^{r}$ explicitly w.r.t these sets, sum up, and arrive at a contradiction to \eqref{eq:claimbycontradiction}.
\begin{lemma}
Assume $p > 2$. There exists $\delta_0 > 0$ such that for any $\tnn \in \N$, any small $\eps \in (0,1)$, and any $\delta  \in (0,\delta_0)$ we have
 \begin{equation}\label{eq:up}
\begin{split}
 &\int_{\B^2} \abs{|\nabla u(x)|^{p-2} \nabla u(x) -\nabla^\perp v(x)}^{\frac{r}{p-1}}\\
 \aleq& \sum_{\tni=2}^{\tnn} \brac{\frac{\delta^{r}}{\eps^{r\frac{p-2}{p-1}}} |\Omega_{B_{\tni}}| + \eps^{\frac{r}{p-1}}\int_{\Omega_{B_{\tni}}} |\nabla u(x)|^{r}}\\
 &+\sum_{\tni=2}^{\tnn} \brac{\frac{\delta^{r}}{\eps^{r\frac{p-2}{p-1}}} |\Omega_{-B_{\tni}}| + \eps^{\frac{r}{p-1}}\int_{\Omega_{-B_{\tni}}} |\nabla u(x)|^{r}}\\
 &+\sum_{\tni=2}^{\tnn} \brac{\frac{\delta^{r}}{\eps^{r\frac{p-2}{p-1}}} |\Omega_{C_{\tni}}| + \eps^{\frac{r}{p-1}}\int_{\Omega_{C_{\tni}}} |\nabla u(x)|^{r}}\\
 &+\sum_{\tni=2}^{\tnn} \brac{\frac{\delta^{r}}{\eps^{r\frac{p-2}{p-1}}} |\Omega_{-C_{\tni}}| + \eps^{\frac{r}{p-1}}\int_{\Omega_{-C_{\tni}}} |\nabla u(x)|^{r}}\\
 &+\tnn^{r-\bar{q}}
 \end{split}
 \end{equation}
\end{lemma}
\begin{proof}
Fix any $\tni$. By the observation \eqref{eq:nablauinKp}, for any $x \in \Omega_{B_{\tni}}$ there are $U(x)$, $V(x) \in \R^2$, $|U(x)|, |V(x)| < \delta$ such that
\[
 |\nabla u(x)+U(x)|^{p-2} \brac{\nabla u(x) +U(x)} - \nabla^\perp v(x) - V^\perp(x) = 0.
\]
Consequently,
\[
 |\nabla u(x)|^{p-2} \nabla u(x) -\nabla^\perp v(x) =  V^\perp(x) + |\nabla u(x)|^{p-2} \nabla u(x)- |\nabla u(x)+U(x)|^{p-2} \brac{\nabla u(x) +U(x)},
\]
and since $p>2$ we have
\[
 \abs{|\nabla u(x)|^{p-2} \nabla u(x)- |\nabla u(x)+U(x)|^{p-2} \brac{\nabla u(x) +U}}\aleq \int_0^1 |\nabla u(x) + s U(x)|^{p-2} ds\, |U(x)|
\]
We have that $|\nabla u(x)| \aeq_{b} \tni$ (and $\tni \geq 1$) so in particular if we choose $\delta$ small enough (not depending on $\tnn$!) we have
\begin{equation}\label{eq:estpl2}
 \abs{|\nabla u(x)|^{p-2} \nabla u(x) -\nabla^\perp v(x)} \aleq \delta |\nabla u(x)|^{p-2} \quad \text{for any $x \in \Omega_{B_{\tni}}$}.
\end{equation}
In particular, from Young's inequality for any $\eps > 0$
\[
 \abs{|\nabla u(x)|^{p-2} \nabla u(x) -\nabla^\perp v(x)} \aleq \frac{\delta^{p-1}}{\eps^{p-2}}  + \eps |\nabla u(x)|^{p-1}.
\]

We conclude that for any $r\geq p-1$ and any $\eps > 0$ we have
\[
 \int_{\Omega_{\pm B_{\tni}}} \abs{|\nabla u(x)|^{p-2} \nabla u(x) -\nabla^\perp v(x)}^{\frac{r}{p-1}} \aleq \frac{\delta^{r}}{\eps^{r\frac{p-2}{p-1}}} |\Omega_{\pm B_{\tni}}| + \eps^{\frac{r}{p-1}}\int_{\Omega_{\pm B_{\tni}}} |\nabla u(x)|^{r}
\]
Similarly
\[
 \int_{\Omega_{\pm C_{\tni}}} \abs{|\nabla u(x)|^{p-2} \nabla u(x) -\nabla^\perp v(x)}^{\frac{r}{p-1}} \aleq \frac{\delta^{r}}{\eps^{r\frac{p-2}{p-1}}} |\Omega_{\pm C_{\tni}}| + \eps^{\frac{r}{p-1}}\int_{\Omega_{\pm C_{\tni}}} |\nabla u(x)|^{r}
\]

It remains to estimate
\[
\begin{split}
 &\int_{\Omega_{\pm A_{\tnn}}} \abs{|\nabla u(x)|^{p-2} \nabla u(x) -\nabla^\perp v(x)}^{\frac{r}{p-1}}\\
 \end{split}
 \]
Recall that
\[
 A_\tnn = \left ( \begin{array}{cc}
                b\tnn & 0\\
                0 & \tnn^{p-1}
               \end{array}
\right ).
\]
Thus, we make a rough estimate, and with the help of \eqref{eq:Annguy}
\[
 \int_{\Omega_{\pm A_{\tnn}}} \abs{|\nabla u(x)|^{p-2} \nabla u(x) -\nabla^\perp v(x)}^{\frac{r}{p-1}}
 \aleq_{b} \bar{\Gamma}_{\tnn} \tnn^r  \aleq \tnn^{r-\bar{q}}.
 \]

 Combining the previous estimates we have shown \eqref{eq:up}.
\end{proof}

On the other hand, we clearly have
\begin{lemma}
\begin{equation}\label{eq:down}
\begin{split}
 \int_{\B^2} |\nabla u|^r \geq& \sum_{\tni=2}^{\tnn} \int_{\Omega_{B_{\tni}} \dot{\cup} \Omega_{-B_{\tni}} \dot{\cup} \Omega_{C_{\tni}} \dot{\cup} \Omega_{-C_{\tni}}} |\nabla u|^{r}.
 \end{split}
 \end{equation}
\end{lemma}

\begin{proof}[Conclusion of the proof of \Cref{th:main} for $p>2$]
Since we assume \eqref{eq:claimbycontradiction}, the estimates \eqref{eq:up} and \eqref{eq:down} imply
\[
\begin{split}
 &\sum_{\tni=2}^{\tnn} \int_{\Omega_{B_{\tni}} \dot{\cup} \Omega_{-B_{\tni}} \dot{\cup} \Omega_{C_{\tni}} \dot{\cup} \Omega_{-C_{\tni}}} |\nabla u|^{r}\\
 \aleq& \sum_{\tni=2}^{\tnn} \Lambda \frac{\delta^{r}}{\eps^{r\frac{p-2}{p-1}}} \brac{|\Omega_{B_{\tni}}|+|\Omega_{-B_{\tni}}|+|\Omega_{C_{\tni}}|+|\Omega_{-C_{\tni}}|}\\
 &+\sum_{\tni=2}^{\tnn} \Lambda  \eps^{\frac{r}{p-1}}\int_{\Omega_{B_{\tni}} \dot{\cup} \Omega_{-B_{\tni}} \dot{\cup} \Omega_{C_{\tni}} \dot{\cup} \Omega_{-C_{\tni}}} |\nabla u(x)|^{r}\\
 &+\Lambda \tnn^{r-\bar{q}}.
 \end{split}
\]
Since $r>p-1$, we choose $\eps$ so small such that we can absorb the $\Lambda \eps^{\frac{r}{p-1}}$-terms to arrive at
\[
\begin{split}
 &\sum_{\tni=2}^{\tnn} \int_{\Omega_{B_{\tni}} \dot{\cup} \Omega_{-B_{\tni}} \dot{\cup} \Omega_{C_{\tni}} \dot{\cup} \Omega_{-C_{\tni}}} |\nabla u|^{r}\\
 \aleq& \sum_{\tni=2}^{\tnn} \Lambda \frac{\delta^{r}}{\eps^{r\frac{p-2}{p-1}}} \brac{|\Omega_{B_{\tni}}|+|\Omega_{-B_{\tni}}|+|\Omega_{C_{\tni}}|+|\Omega_{-C_{\tni}}|} +\Lambda \tnn^{r-\bar{q}}.
 \end{split}
\]
But observe that on $\Omega_{\pm B_{\tni}}$ and $\Omega_{\pm C_{\tni}}$ we have $|\nabla u| \aeq_{b} \tni \geq 1$, so that we arrive at
\[
\begin{split}
 &\sum_{\tni=2}^{\tnn} \brac{|\Omega_{B_{\tni}}|+|\Omega_{-B_{\tni}}|+|\Omega_{C_{\tni}}|+|\Omega_{-C_{\tni}}|} \\
 \aleq& \sum_{\tni=2}^{\tnn} \Lambda \frac{\delta^{r}}{\eps^{r\frac{p-2}{p-1}}} \brac{|\Omega_{B_{\tni}}|+|\Omega_{-B_{\tni}}|+|\Omega_{C_{\tni}}|+|\Omega_{-C_{\tni}}|} +\Lambda \tnn^{r-\bar{q}}.
 \end{split}
\]Choose $\delta$ even smaller (depending on $\eps$ and $\Lambda$), then we can absorb the remaining $\Lambda$ term and have shown
\[
 \sum_{\tni=2}^{\tnn} \brac{|\Omega_{B_{\tni}}|+|\Omega_{-B_{\tni}}|+|\Omega_{C_{\tni}}|+|\Omega_{-C_{\tni}}|}  \\
 \aleq\Lambda \tnn^{r-\bar{q}}.
\]
Now
\[
 \sum_{\tni=2}^{\tnn} \brac{|\Omega_{B_{\tni}}|+|\Omega_{-B_{\tni}}|+|\Omega_{C_{\tni}}|+|\Omega_{-C_{\tni}}|}  = |\B^2 \setminus \brac{\Omega_{A_{\tnn}} \dot{\cup} \Omega_{-A_{\tnn}}} | \overset{\eqref{eq:Annguy}}{=} |\B^2| (1-2\overline{\Gamma}_{N}) \ageq |\B|^2 (1-c\tnn^{-\bar{q}}).
\]
Thus, the estimate we have is
\[
\begin{split}
 (1-c\tnn^{-\bar{q}}) \leq C \Lambda \tnn^{r-\bar{q}},
\end{split}
\]
where the constants $c,C$ depend on $\Lambda$, $p$, $b$, but they are independent of $\tnn$. Since $r \in (p-1,\bar{q})$ the above is a contradiction for large $\tnn$. Thus \eqref{eq:claimbycontradiction} must be false and thus \eqref{eq:claim} is proven, at least under the assumption for $p>2$. \end{proof}

\begin{remark}
The above proof assumes $p > 2$. If $p \in (1,2)$ we simply need to adapt \eqref{eq:estpl2} into
\[
 \abs{|\nabla u(x)|^{p-2} \nabla u(x) -\nabla^\perp v(x)} \aleq \delta^{\frac{p-1}{2}} |\nabla u(x)|^{\frac{p-1}{2}} \quad \text{for any $x \in \Omega_{B_{\tni}}$, $p \in (1,2)$}.
\]
and continue with an analogous argument.
\end{remark}

\bibliographystyle{abbrv}
\bibliography{bib}

\end{document}